\documentclass[12pt]{amsart}

\usepackage[foot]{amsaddr}
\usepackage{graphicx}
\usepackage{mathptmx} 
\usepackage{amssymb}
\usepackage{amsmath}
\usepackage{amsthm}
\usepackage{mathtools}
\usepackage[dvipsnames]{xcolor}
\usepackage{marvosym}
\usepackage{a4wide}
\usepackage[normalem]{ulem}
\newcommand{\stkout}[1]{\ifmmode\text{\sout{\ensuremath{#1}}}\else\sout{#1}\fi}
\usepackage{enumerate}
\usepackage{url}
\usepackage{csquotes}

\usepackage[mathscr]{euscript}

\theoremstyle{plain}
\newtheorem{theorem}{Theorem}[section]
\newtheorem{lemma}[theorem]{Lemma} 
\newtheorem{corollary}[theorem]{Corollary}
\newtheorem{proposition}[theorem]{Proposition}

\newtheorem{fact}[theorem]{Fact}

\theoremstyle{definition}

\newcommand{\cat}{\text{CAT}}

\title{EL-Shellability of the poset of ranked cactuses}

\author{Vincent Moulton}
\address[V. Moulton]{School of Computing Sciences, University of East Anglia, UK}
\email{v.moulton@uea.ac.uk}
\author{Andreas Spillner}
\address[A. Spillner]{Merseburg University of Applied Sciences, Germany}
\email{andreas.spillner@hs-merseburg.de}
\author{Antonia Stavemann}
\address[A. Stavemann]{Interdisciplinary Center for Scientific Computing, Heidelberg University + Merseburg University of Applied Sciences, Germany}
\email{antonia.stavemann@hs-merseburg.de}

\date{\today}

\begin{document}

\begin{abstract}
Recently the \emph{poset of ranked cactuses}~\((\mathfrak{P}(X),\preceq)\) was
introduced. For a finite set $X$, 
this poset consists of a set~\(\mathfrak{P}(X)\)
of certain collections of ordered pairs of subsets of~\(X\) together
with an ordering $\preceq$ that is similar to the refinement ordering of partitions of a finite set. In addition, 
the maximal chains in this poset correspond to
\emph{binary ranked cactuses}, a fact which can be used to
construct the so-called \emph{space of equidistant cactuses}.
In this paper, we show that the poset of ranked cactuses is EL-shellable.
As a consequence we also show that 
the proper part of the link of the origin of the 
space of equidistant cactuses has the homotopy type of a wedge of spheres.
\end{abstract}

\maketitle

\section{Introduction}

Spaces of phylogenetic trees, or \emph{tree spaces}, are 
well-studied objects in phylogenetics~\cite{ard-kli-06a,BHV01,GD16,moulton2004peeling}. The geometry of these spaces is important
because it can be used to define \emph{metrics} on
trees which are useful, e.g., in tree-construction
algorithms or consensus tree approaches \cite{st2017shape}.
Interestingly, the geometry of many of 
these tree spaces can be understood through combinatorial 
properties of associated partially ordered sets of trees or related 
structures; see e.g.~\cite{gil-etal-08a,stadnyk2022edge,tra-zie-98a}. 
For example, given a finite, non-empty set $X$ with~\(n\) elements,
the orthant space of \emph{ultrametric} (or equidistant)
\emph{trees} with leaf set~\(X\) defined in~\cite{GD16} 
(see also \cite{ard-kli-06a}) can be
obtained from the order complex of the poset of partitions
of~\(X\) under refinement. This can be shown
using the fact that the maximal chains in the poset of partitions
of~\(X\) are in one-to-one correspondence with rooted binary
trees with leaf set~\(X\) and whose interior vertices are
ranked (see Figure~\ref{fig:poset:partitions} for 
the case $n=4$).
Moreover, through this connection, the fact
that the space of ultrametric trees is a so-called $\cat(0)$-space~\cite{GD16}  also follows -- see e.g. \cite[Section 4]{huber2024space}. 

\begin{figure}
\centering
\includegraphics[scale=0.88]{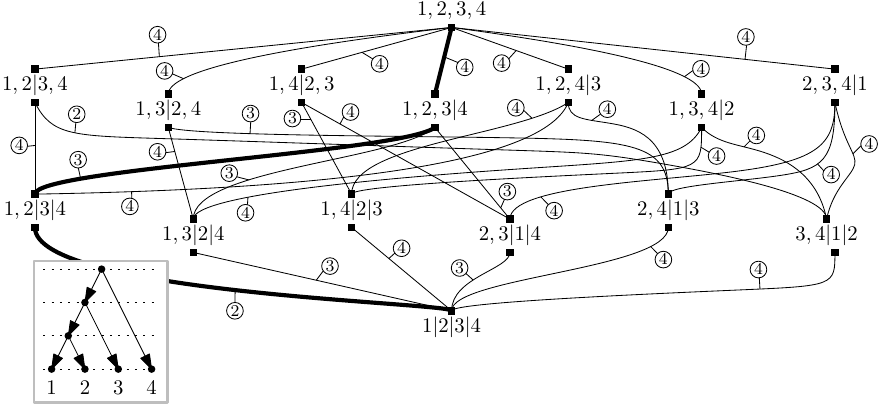}
\caption{The Hasse diagram of the poset of partitions of \(X=\{1,2,3,4\}\).
To avoid expressions of many nested brackets we write, for example,
\(1,3|2|4\) for the partition of~\(X\) into the subsets
\(\{1,3\}\), \(\{2\}\) and \(\{4\}\).
The bold maximal chain corresponds to the binary ranked tree with
leaf set~\(X\) depicted in the gray box. To every edge in the 
Hasse diagram is assigned an integer label which is depicted in
a small circle. The sequence of labels of the bold maximal chain
from bottom to top is \((2,3,4)\). The bold maximal chain is the
only maximal chain with a strictly increasing sequence of labels.
Moreover, this sequence is lexicographically strictly smaller than
the sequence of labels of any other maximal chain.}
\label{fig:poset:partitions}
\end{figure}

Recently, there has been increasing interest in understanding properties of
certain graphs generalizing phylogenetic trees called \emph{phylogenetic
networks} (see e.g. \cite{kong2022classes}). Such
networks permit the representation of evolutionary process such
as hybridization and recombination that do not fit on a tree.
This has naturally led to an emerging theory of \emph{network spaces}, 
continuous spaces of phylogenetic networks that can be thought of as generalizations of tree spaces (see e.g. \cite{DP17a,huber2024space,moulton2025spaces}).
The motivation underlying this paper is to better
understand the space of \emph{equidistant cactuses} with leaf set~\(X\)
introduced in~\cite{huber2024space} and, more specifically,
the \emph{poset of ranked cactuses}~\((\mathfrak{P}(X),\preceq)\) that was employed in that paper to construct equidistant cactus space. 
This poset consists of the set~\(\mathfrak{P}(X)\) of certain
collections of ordered pairs of subsets of~\(X\) together
with an ordering $\preceq$ that is similar
to the refinement ordering of partitions, and
maximal chains in this poset correspond to
binary ranked phylogenetic cactuses
(see Figure~\ref{fig:example:binary:cactus} for an example).
Note that the poset of partitions of~\(X\) is isomorphically embedded
in~\((\mathfrak{P}(X),\preceq)\)~\cite[Cor.~1]{huber2024space}.

In this paper, we shall prove that the poset \((\mathfrak{P}(X),\preceq)\)
is \emph{edge-lexicographic (EL-) shellable}.
Poset shellability is an important property that can be used to determine 
the topology of posets -- see \cite{wachs07} for an overview.
EL-shellability involves showing that there exists an integer labeling~\(\lambda\)
of the edges in the Hasse diagram of a poset such that every
interval of the poset has a unique maximal chain~\(\mathfrak{C}\) 
along which labels are strictly increasing, and such that
the sequence of labels along~\(\mathfrak{C}\)
is lexicographically strictly smaller than the
sequence of labels along any other maximal chain in the
same interval. Note that it is known~\cite[Sec.~3.2.2.]{wachs07}
that the poset of partitions of~\(X\) admits such a labeling
as illustrated in Figure~\ref{fig:poset:partitions}.

\begin{figure}
\centering
\includegraphics[scale=1.0]{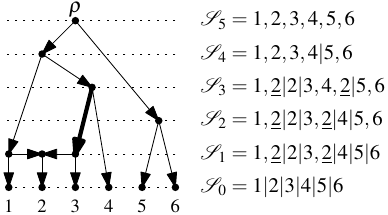}
\caption{A binary ranked cactus with leaf set~\(X=\{1,2,\dots,6\}\).
The dotted horizontal lines indicate the ranks of the vertices.
For each rank~\(i\), \(0 \leq i \leq 5\),
the associated collection \(\mathcal{S}_i\) of ordered pairs of
subsets of~\(X\) is given. 
To avoid expressions of many nested brackets we continue to
use the notation introduced for partitions and mark the elements
in the second entry of an ordered pair of subsets of~\(X\) by
underlining them. The bold arc, for example corresponds to the
ordered pair \((\{3\},\{2\})\) in~\(\mathcal{S}_2\),
which is denoted as~\(3,\underline{2}\).
This ordered pair records the fact that all directed
paths from the root~\(\rho\) to leaf~\(3\) must include the
bold arc whereas leaf~\(2\) can be reached from~\(\rho\)
both by a directed path including the bold arc and by a directed path
avoiding it. The sequence \(\mathcal{S}_0,\mathcal{S}_1,\dots,\mathcal{S}_5\)
then encodes the whole binary ranked cactus.}
\label{fig:example:binary:cactus}
\end{figure} 

The rest of this paper is structured as follows.
In Section~\ref{sec:prelim} we recall some key concepts concerning posets.
Then in Section~\ref{sec:properties:poset} we
formally define the poset~\((\mathfrak{P}(X),\preceq)\)
of ranked cactuses
and, after establishing some of its basic properties, in
Section~\ref{sec:poset:shellable} we show that
this poset is EL-shellable. In Section~\ref{sec:conclusion}
we show that since
\((\mathfrak{P}(X),\preceq)\) is EL-shellable,
it immediately follows that the proper part of 
the link of the origin of the orthant space
of equidistant cactuses has the homotopy type of a wedge of spheres,
thus answering a question posed in~\cite[Section 9]{huber2024space}.
We conclude by pointing out a potential direction for future work
concerning more general ranked networks.

\section{Preliminaries}
\label{sec:prelim}

In this section, we recall some key concepts
related to finite posets and introduce the notation
used for them in this paper. These standard concepts can
be found in e.g. \cite{wachs07}. Let \((M,\leq)\) be a finite poset.
\(m \in M\) is a \emph{maximum} (\emph{minimum})
of \((M,\leq)\) if \(e \leq m\) (\(m \leq e\)) for all \(e \in M\).
Note that, if it exists, the maximum (minimum) of \((M,\leq)\) is unique.
A poset is \emph{bounded} if it has a maximum and a minimum.
The \emph{proper part} \((\overline{M},\leq)\) of
a bounded poset \((M,\leq)\) is formed
by removing the maximum and minimum from~\(M\).

A non-empty subset \(\mathfrak{C} \subseteq M\) is a \emph{chain}
in \((M,\leq)\) if the elements of \(\mathfrak{C}\) can be
numbered such that
\[e_0 \leq e_1 \leq \dots \leq e_k,\]
where \(k = |\mathfrak{C}|-1\) is called the \emph{length} of
the chain. A chain \(\mathfrak{C}\) in \((M,\leq)\) is
\emph{maximal} if \(\mathfrak{C} \cup \{a\}\) is not a chain
for all \(a \in M - \mathfrak{C}\). A bounded poset \((M,\leq)\)
is \emph{graded} (also called \emph{pure})
if all maximal chains in \((M,\leq)\) have
the same length.

An element \(u \in M\) is a \emph{least upper bound} of \(a\) and \(b\)
if \(a \leq u\), \(b \leq u\) and, for all \(u' \in M\)
with \(a \leq u'\) and \(b \leq u'\), \(u \leq u'\).
If \(a\) and \(b\) have a least upper bound it is unique.
The \emph{greatest lower bound} of \(a\) and \(b\) is
defined analogously.
A poset \((M,\leq)\) is a \emph{lattice}
if for all \(a,b \in M\) the least
upper bound and the greatest lower bound exist.

For two distinct
\(a,b \in M\) with \(a \leq b\), the \emph{interval} \([a,b]\) in~\((M,\leq)\)
consists of all \(e \in M\) with \(a \leq e \leq b\). For each 
interval \(\mathfrak{I}\) in~\((M,\leq)\) we will refer to a
chain~\(\mathfrak{C}\) in~\((M,\leq)\) with
\(\mathfrak{C} \subseteq \mathfrak{I}\) 
as a \emph{chain in}~\(\mathfrak{I}\). A chain \(\mathfrak{C}\)
in~\(\mathfrak{I}\) is \emph{maximal in}~\(\mathfrak{I}\) if
\(\mathfrak{C} \cup \{e\}\) is not a chain in~\(\mathfrak{I}\)
for all \(e \in \mathfrak{I} - \mathfrak{C}\).

The \emph{edge set} \(\mathfrak{E}_{(M,\leq)}\) of~\((M,\leq)\)
consists of those ordered pairs \((a,b) \in M \times M\)
with \(|[a,b]|=2\). \(\mathfrak{E}_{(M,\leq)}\) is also referred to as the
\emph{covering relation} of \((M,\leq)\). The elements of
\(\mathfrak{E}_{(M,\leq)}\) are called \emph{edges} of~\((M,\leq)\).
A poset~\((M,\leq)\) is \emph{thin}
if for all edges \((a,b),(b,c) \in \mathfrak{E}_{(M,\leq)}\)
the interval \([a,c]\) contains precisely four elements.

An \emph{edge-labeling}
of \((M,\leq)\) is a map \(\lambda: \mathfrak{E}_{(M,\leq)} \rightarrow \mathbb{N}\). The \emph{label vector} of a chain~\(\mathfrak{C}\)
of length~\(k \geq 1\) in~\((M,\leq)\) with elements
\(e_0 \leq e_1 \leq \dots \leq e_k\) is the vector
\[\Lambda_{\lambda}(\mathfrak{C})
= (\lambda(e_0,e_1),\lambda(e_1,e_2), \dots , \lambda(e_{k-1},e_k))
\in \mathbb{N}^k.\]
A chain~\(\mathfrak{C}\) of length \(k \geq 1\) in \((M,\leq)\)
is \emph{increasing} with respect to an edge-labeling~\(\lambda\)
of~\((M,\leq)\) if the entries of the labeling vector
\(\Lambda_{\lambda}(\mathfrak{C})\) are strictly increasing
from the first to the last entry.
To compare vectors \(u=(u_1,\dots,u_k),v=(v_1,\dots,v_k) \in \mathbb{N}^k\)
we use the \emph{lexicographic ordering}, that is,
\(u <_{\text{lex}} v\) if \(u \neq v\) and
\(u_i < v_i\) for the smallest \(1 \leq i \leq k\) with
\(u_i \neq v_i\).

A finite, bounded and graded poset~\((M,\leq)\) is 
\emph{edge-lexicographic shellable} (EL-shellable, for short;
see \cite[Sec.~3.2]{wachs07})
if there exists an edge-labeling~\(\lambda\) 
of~\((M,\leq)\) such that
\begin{itemize}
\item[(EL1)]
for every interval \(\mathfrak{I}\)
of \((M,\leq)\) with \(|\mathfrak{I}| \geq 2\),
there is precisely one maximal chain
\(\mathfrak{C}_{\mathfrak{I}}\) in \(\mathfrak{I}\) that is increasing
with respect to~\(\lambda\), and 
\item[(EL2)]
the labeling vector \(\Lambda_{\lambda}(\mathfrak{C}_{\mathfrak{I}})\)
is lexicographically 
strictly smaller than the labeling vector of any other
maximal chain in~\(\mathfrak{I}\).
\end{itemize}

Note that EL-shellability of a 
finite, bounded and graded poset~\((M,\leq)\) has important consequences
for its combinatorial and topological properties. In particular,
using the terminology and notation in~\cite[Section 1.1]{wachs07},
EL-shellability of~\((M,\leq)\) implies 
that the \emph{order complex}~$\Delta(\overline{M})$ of its proper part,
that is, the abstract simplicial complex whose vertex
set is~$\overline{M}$ and whose
simplices are the chains in \((\overline{M},\leq)\) 
has the homotopy type of a wedge of spheres~\cite{bjorner1996nonpure}  
(see also~\cite[Sec.~3.1 and~4.1]{wachs07} for some other
interesting consequences of shellability of simplicial complexes).

\section{Basic properties of the poset $(\mathfrak{P}(X),\preceq)$}
\label{sec:properties:poset}

We begin by formally defining
the poset~\((\mathfrak{P}(X),\preceq)\) as introduced in~\cite{huber2024space}.
Let \(X\) be a finite non-empty set and put \(n = |X|\).
A \emph{set pair} on~\(X\) is an ordered pair of subsets of~\(X\).
A \emph{set pair system} on~\(X\)
is a collection of set pairs $(S,H)$ on~\(X\) 
with \(S \neq \emptyset\) and \(S \cap H = \emptyset\).
Without loss of generality
we assume from now on that \(X=\{1,2,\dots,n\}\) and,
when giving examples of set pair systems, we use the
notation introduced in Figure~\ref{fig:example:binary:cactus}
to avoid nested brackets.

For a set pair system~\(\mathcal{S}\) on~\(X\), put
\[\mathcal{P}(\mathcal{S}) = \{S : (S,H) \in \mathcal{S}\}
\quad \text{and} \quad
\mathcal{H}(\mathcal{S}) = \{H : (S,H) \in \mathcal{S}, H \neq \emptyset\}.\]
A set pair system \(\mathcal{S}\) on \(X\) is \emph{partition-like}~if
\begin{itemize}
\setlength{\itemindent}{15pt}
\item[(PL1)]
\(\mathcal{P}(\mathcal{S})\) is a partition of~\(X\),
\item[(PL2)]
for all \((S,H),(S',H') \in \mathcal{S}\) with \((S,H) \neq (S',H')\)
we have \(S \neq S'\), and
\item[(PL3)]
for all \((S,H) \in \mathcal{S}\) with \(H \neq \emptyset\) we have
\((H,\emptyset) \in \mathcal{S}\) and
there exists precisely one \((\overline{S},\overline{H}) \in \mathcal{S}\)
with \((\overline{S},\overline{H}) \neq (S,H)\) and \(H=\overline{H}\).
\end{itemize}
Note that~(PL2) implies that $\lvert \mathcal{S} \rvert = \lvert \mathcal{P}(\mathcal{S}) \rvert$.
We denote by~\(\mathfrak{P}(X)\) the set of
partition-like set pair systems on the set~\(X\).
In addition, for all \(\mathcal{S} \in \mathfrak{P}(X)\) and
all \((S,H) \in \mathcal{S}\) with \(H \neq \emptyset\),
we denote by \((\overline{S},H)\) the unique
set pair in~\(\mathcal{S}-\{(S,H)\}\) with second entry~\(H\)
that must exist by~(PL3). 

To define the partial ordering on~\(\mathfrak{P}(X)\),
we first put \((S_1,H_1) \leq (S_2,H_2)\) for
set pairs \((S_1,H_1)\) and  \((S_2,H_2)\) on~\(X\)
if either \((S_1,H_1) = (S_2,H_2)\) or \((S_1,H_1) \neq (S_2,H_2)\) and one the
following holds:
\begin{itemize}
\item
\(S_1 \cup H_1 \subseteq S_2\)
\item
\(S_1 \subsetneq S_2\) and \(H_1 = H_2 \neq \emptyset\)
\end{itemize}
Moreover, we write \((S_1,H_1) < (S_2,H_2)\) if \((S_1,H_1) \leq (S_2,H_2)\)
and the set pairs \((S_1,H_1)\) and \((S_2,H_2)\) are distinct.
Note that in~\cite{huber2024space} the definition of~\(\leq\)
for set pairs is slightly more general. We can simplify it here
because in the following we only apply it to partition-like set pair systems.
For \(\mathcal{S}_1,\mathcal{S}_2 \in \mathfrak{P}(X)\)
we put \(\mathcal{S}_1 \preceq \mathcal{S}_2\) if 
\begin{itemize}
\setlength{\itemindent}{15pt}
\item[(SP1)]
for all \((S_1,H_1) \in \mathcal{S}_1\) there exists some
\((S_2,H_2) \in \mathcal{S}_2\) with \((S_1,H_1) \leq (S_2,H_2)\), and
\item[(SP2)]
for all \((S_2,H_2) \in \mathcal{S}_2\) with \(H_2 \neq \emptyset\),
if there exists some \((S_1,H_1) \in \mathcal{S}_1\) with
\(H_1 = H_2\), then there exists such a \((S_1,H_1)\) with
\((S_1,H_1) \leq (S_2,H_2)\).
\end{itemize}
Again, we write \(\mathcal{S}_1 \prec \mathcal{S}_2\) if
\(\mathcal{S}_1 \preceq \mathcal{S}_2\)
and~\(\mathcal{S}_1 \neq \mathcal{S}_2\).
We remark that~(SP2) is a technical requirement that ensures
that elements cannot switch between \(S\) and \(\overline{S}\)
in two set pairs with the same non-empty second entry~\(H\) as
in \(\mathcal{S}_1 = 1,\underline{2}|2|3,\underline{2}|4\) and
\(\mathcal{S}_2 = 1,3,\underline{2}|2|4,\underline{2}\),
which satisfy~(SP1) but not~(SP2).

It is shown in~\cite[Prop.~1]{huber2024space} that
\((\mathfrak{P}(X),\preceq)\) is a bounded graded poset.
The minimum of \((\mathfrak{P}(X),\preceq)\) is
\(\{(\{x\},\emptyset): x \in X\}\) and the 
maximum is \(\{(X,\emptyset)\}\). All maximal chains
in~\((\mathfrak{P}(X),\preceq)\) have length~\(n-1\).
As an example, we consider the collection
\(\{\mathcal{S}_0,\mathcal{S}_1,\dots,\mathcal{S}_5\}\)
of set pair systems given in Figure~\ref{fig:example:binary:cactus}.
This collection is a maximal chain
in~\((\mathfrak{P}(\{1,\dots,6\}),\preceq)\).

As illustrated in Figure~\ref{fig:poset:partitions},
in the poset of partitions of~\(X\) edges correspond
to unions of precisely two disjoint subsets of~\(X\) in a partition of~\(X\)
to form a coarser partition of~\(X\).
Intuitively, the poset~\((\mathfrak{P}(X),\preceq)\)
allows to postpone such pairwise unions and mark the
involved subsets by forming set pairs with a non-empty
second entry. Next we capture this intuition more precisely for later use
(cf.~\cite[Lem.~6]{huber2024space}):

\begin{fact}
\label{fact:edges:p:x}
Let \((\mathcal{S},\mathcal{S}')\) be an edge
of~\((\mathfrak{P}(X),\preceq)\). Then precisely one of the
following must hold:
\begin{itemize}
\item[(1)]
There exist three pairwise distinct
\((S_1,H_1),(S_2,H_2),(S_3,H_3) \in \mathcal{S}\)
with \(H_1=H_2=S_3\) and \(H_3=\emptyset\)
such that
\[\mathcal{S}' = (\mathcal{S} - \{(S_1,H_1),(S_2,H_2),(S_3,H_3) \})
\cup \{(S_1 \cup S_2 \cup S_3,\emptyset)\}.\]
Intuitively, this performs the postponed union of~\(S_1\) and~\(S_2\)
that was marked with~\(S_3\).
We denote this by \(\mathcal{S} \vdash_{(1)} \mathcal{S}'\).
\item[(2)]
There exist two distinct \((S_1,H_1),(S_2,H_2) \in \mathcal{S}\)
with \(H_1=\emptyset\), \(H_2 \neq \emptyset\) and 
\((S_1 \cup S_2) \cap H_3 = \emptyset\) for all
\((S_3,H_3) \in \mathcal{S}\)
such that
\[\mathcal{S}' = (\mathcal{S} - \{(S_1,H_1),(S_2,H_2)\})
\cup \{(S_1 \cup S_2,H_2)\}.\]
We denote this by \(\mathcal{S} \vdash_{(2)} \mathcal{S}'\).
\item[(3)]
There exist two distinct \((S_1,H_1),(S_2,H_2) \in \mathcal{S}\)
with \(H_1=H_2=\emptyset\) and 
\((S_1 \cup S_2) \cap H_3 = \emptyset\) for all
\((S_3,H_3) \in \mathcal{S}\)
such that
\[\mathcal{S}' = (\mathcal{S} - \{(S_1,H_1),(S_2,H_2)\})
\cup \{(S_1 \cup S_2,\emptyset)\}.\]
We denote this by \(\mathcal{S} \vdash_{(3)} \mathcal{S}'\).
\item[(4)]
There exist three pairwise distinct
\((S_1,H_1),(S_2,H_2),(S_3,H_3) \in \mathcal{S}\)
with \(H_1=H_2=H_3=\emptyset\)
and \((S_1 \cup S_2 \cup S_3) \cap H_4 = \emptyset\) for all
\((S_4,H_4) \in \mathcal{S}\)
such that
\[\mathcal{S}' = (\mathcal{S} - \{(S_1,H_1),(S_2,H_2)\})
\cup \{(S_1,S_3),(S_2,S_3)\}.\]
Intuitively, this postpones the union of~\(S_1\) and~\(S_2\),
marking them with~\(S_3\).
We denote this by \(\mathcal{S} \vdash_{(4)} \mathcal{S}'\).
\end{itemize}
In addition, we write \(\mathcal{S} \vdash \mathcal{S}'\)
if \(\mathcal{S} \vdash_{(i)} \mathcal{S}'\) for some
\(i \in \{1,2,3,4\}\).
\end{fact}  

Note that \(\vdash\) is the covering relation
of~\((\mathfrak{P}(X),\preceq)\).
To illustrate~Fact~\ref{fact:edges:p:x}, we consider again
the maximal chain in~\((\mathfrak{P}(\{1,\dots,6\}),\preceq)\)
given in Figure~\ref{fig:example:binary:cactus} for which we have
\[\mathcal{S}_0 \vdash_{(4)} \mathcal{S}_1
\vdash_{(3)} \mathcal{S}_2
\vdash_{(2)} \mathcal{S}_3
\vdash_{(1)} \mathcal{S}_4
\vdash_{(3)} \mathcal{S}_5.\]

In Section~\ref{sec:poset:shellable} we will construct an
edge-labeling of~\((\mathfrak{P}(X),\preceq)\)
that relies on Fact~\ref{fact:edges:p:x}.
We conclude this section looking at two properties 
of the poset~\((\mathfrak{P}(X),\preceq)\). First note that,
in contrast to the \emph{Tuffley poset} considered
in~\cite{moulton2004peeling} which is related to
trees with leaf set~\(X\), the poset~\((\mathfrak{P}(X),\preceq)\)
is not thin, since the poset of partitions of~\(X\),
which is isomorphically embedded in~\((\mathfrak{P}(X),\preceq)\),
is known to be not thin (as can be seen in the example in
Figure~\ref{fig:poset:partitions}).
Moreover, in contrast to the poset of partitions of~\(X\),
\((\mathfrak{P}(X),\preceq)\) is, in general, not a lattice.

\begin{proposition}
\label{prop:not:a:lattice}
For \(n \geq 5\) the poset
\((\mathfrak{P}(X),\preceq)\) is not a lattice.
\end{proposition}

\begin{proof}
For \(n=5\) consider the set pair systems
\[\mathcal{S} = 1|2,\underline{5}|5|3,\underline{5}|4 
\quad \text{and} \quad
\mathcal{S}' = 2|1,\underline{5}|5|4,\underline{5}|3\]
in \(\mathfrak{P}(X)\).
Put
\[\mathcal{S}_1'' = 1,2,\underline{5}|5|3,4,\underline{5}
\quad \text{and} \quad
\mathcal{S}_2'' = 1,3,\underline{5}|5|2,4,\underline{5}.\]
Then \(\mathcal{S}''_1,\mathcal{S}''_2 \in \mathfrak{P}(X)\)
with \(\mathcal{S} \preceq \mathcal{S}_1''\) and
\(\mathcal{S}' \preceq \mathcal{S}_2''\).
Moreover, it can be checked that there is
no \(\mathcal{S}'' \in \mathfrak{P}(X)\) with
\(\mathcal{S} \vdash \mathcal{S}''\) and
\(\mathcal{S}' \vdash \mathcal{S}''\).
Hence, \(\mathcal{S}\) and \(\mathcal{S}'\) don't have
a least upper bound, implying that
\((\mathfrak{P}(X),\preceq)\) cannot be a lattice.

The counterexample can be generalized to \(n \geq 6\)
by adding the elements \(6,\dots,n\) to those sets
in the counterexample above that contain~5.
\end{proof}

We remark that for \(n \leq 4\) the poset \((\mathfrak{P}(X),\preceq)\) is a lattice. For \(n \in \{2,3\}\) this follows immediately from the
fact that the poset is bounded and all maximal chains have length~\(n-1\).
For \(n=4\) it follows by a case analysis on the structure of those set pair systems \(\mathcal{S} \in \mathfrak{P}(X)\) with 
\(\{(\{x\},\emptyset):x \in X\} \vdash \mathcal{S}\).

\section{Shellability of the poset $(\mathfrak{P}(X),\preceq)$}
\label{sec:poset:shellable}

We now show that the poset of ranked cactuses is EL-shellable.
We first establish a technical lemma that is based on
Fact~\ref{fact:edges:p:x} and which is illustrated in
Figure~\ref{fig:mc:property}.

\begin{lemma}
\label{lem:max:chain:sorted:by:types}
Let \(\mathcal{S}, \mathcal{S}' \in \mathfrak{P}(X)\) 
with \(\mathcal{S} \prec \mathcal{S}'\). Then there
exists a maximal chain \(\mathfrak{C}\) in
\([\mathcal{S},\mathcal{S}']\) with the following
property, where \(k=|\mathfrak{C}|-1\):
\begin{itemize}
\item[(MC)]
The elements of \(\mathfrak{C}\) can be numbered
\(\mathcal{S}_0,\mathcal{S}_1,\dots,\mathcal{S}_k\) such that
\[\mathcal{S}_0 \vdash_{(i_1)} \mathcal{S}_1 \vdash_{(i_2)} \dots
\vdash_{(i_k)} \mathcal{S}_k\]
with \(i_1 \leq i_2 \leq \dots \leq i_k\).
\end{itemize}
Moreover, the sequence \(i_1,i_2,\dots,i_k\) is the same
for all maximal chains in~\([\mathcal{S},\mathcal{S}']\)
satisfying~(MC).
\end{lemma}

\begin{proof}
First we show that there exists a maximal
chain~\(\mathfrak{C}\) in~\([\mathcal{S},\mathcal{S}']\)
satisfying~(MC).
Put \(\mathcal{S}_0 = \mathcal{S}\). For \(1 \leq \ell \leq k\) 
we describe how to construct a suitable set pair system \(\mathcal{S}_{\ell}\)
from \(\mathcal{S}_{\ell-1}\), assuming that
\(\mathcal{S}_{\ell-1} \preceq \mathcal{S}'\).
The construction of~\(\mathfrak{C}\) proceeds in four phases,
where we switch to
the next phase once the condition considered in the current phase
is not (or no longer) met and we never return to a previous phase.\\

\noindent  \underline{Phase 1}: There exists \(H \in \mathcal{H}(\mathcal{S}_{\ell-1}) - \mathcal{H}(\mathcal{S}')\). Let \((S,H),(\overline{S},H),(H,\emptyset) \in \mathcal{S}_{\ell-1}\)
be the three set pairs with one entry equal to~\(H\) that must exist by
the definition of \(\mathcal{H}(\mathcal{S}_{\ell-1})\) and~(PL3).
Put
\[\mathcal{S}_{\ell-1} \vdash_{(1)} \mathcal{S}_{\ell}
= (\mathcal{S}_{\ell-1} - \{(S,H),(\overline{S},H),(H,\emptyset)\})
\cup \{(S \cup \overline{S} \cup H,\emptyset)\}.\]
Since \(H \not \in \mathcal{H}(\mathcal{S}')\) and
\(\mathcal{S}_{\ell-1} \prec \mathcal{S}'\) there must exist
\((S_1,H_1) \in \mathcal{S}'\) with \(S \cup H \subseteq S_1\)
and \((S_2,H_2) \in \mathcal{S}'\) with \(\overline{S} \cup H \subseteq S_2\).
By~(PL1) and the fact that \(H \neq \emptyset\), we have
\(S_1=S_2\) and, thus, \(S \cup \overline{S} \cup H \subseteq S_1\).
Hence, \(\mathcal{S}_{\ell} \preceq \mathcal{S}'\), as required.

At the end of Phase~1, the resulting set pair system
\(\mathcal{S}_{\ell}\) satisfies 
\(\mathcal{H}(\mathcal{S}_{\ell}) \subseteq \mathcal{H}(\mathcal{S}')\)
and this property is maintained through all subsequent phases.\\

\noindent  \underline{Phase 2}: There exist \(H \in \mathcal{H}(\mathcal{S}_{\ell-1}) \cap \mathcal{H}(\mathcal{S}')\), \((S_1,H) \in \mathcal{S}_{\ell-1}\) and
\((S',H) \in \mathcal{S}'\) such that \(S_1 \subsetneq S'\).
Let \(x \in S' - S_1\). By~(PL1) and~(PL2) there exists a unique
\((S_2,H_2) \in \mathcal{S}_{\ell-1}\) with \(x \in S_2\).
Thus, again by~(PL1) and since \(\mathcal{S}_{\ell-1} \prec \mathcal{S}'\),
we must have \((S_2,H_2) < (S',H)\). Assume for a contradiction
that \(H_2 \neq \emptyset\). If \(H_2 \neq H\) we would have
\(S_2 \cup H_2 \subseteq S'\) and, therefore,
\(H_2 \in \mathcal{H}(\mathcal{S}_{\ell-1}) - \mathcal{H}(\mathcal{S}')\),
contradicting the fact that Phase~1 is completed. 
If \(H_2=H\) consider the set pair \((\overline{S}',H) \in \mathcal{S}'\)
that must exist by~(PL3) and the set pair \((S_3,H) \in \mathcal{S}_{\ell-1}\)
with \((S_3,H) \leq (\overline{S}',H)\) that must exist by~(SP2).
Then \(\mathcal{S}_{\ell-1}\) would contain three set pairs with
second entry~\(H\), contradicting~(PL3).
Hence, \(H_2=\emptyset\). Moreover, again since Phase~1 is completed,
\(S_2 \not \in \mathcal{H}(\mathcal{S}_{\ell-1})\).
We put
\[\mathcal{S}_{\ell-1} \vdash_{(2)} \mathcal{S}_{\ell}
= (\mathcal{S}_{\ell-1} - \{(S_1,H),(S_2,\emptyset)\})
\cup \{(S_1 \cup S_2,H)\}.\]
Since \(S_1 \cup S_2 \subseteq S'\), we
have \(\mathcal{S}_{\ell} \preceq \mathcal{S}'\), as required.

At the end of Phase~2, the resulting set pair system
\(\mathcal{S}_{\ell}\) satisfies
\[\{(S,H) \in \mathcal{S}_{\ell}: H \neq \emptyset\}
\subseteq \{(S',H') \in \mathcal{S}': H' \neq \emptyset\}\]
and this property is maintained through all subsequent phases.\\

\noindent \underline{Phase 3}: There exist two distinct \((S_1,\emptyset),(S_2,\emptyset) \in \mathcal{S}_{\ell-1}\) and \((S',H) \in \mathcal{S}'\) with
\(S_1 \cup S_2 \subseteq S'\). Then, since Phase~1 is completed,
\(S_1,S_2 \not \in \mathcal{H}(\mathcal{S}_{\ell-1})\).
We put
\[\mathcal{S}_{\ell-1} \vdash_{(3)} \mathcal{S}_{\ell}
= (\mathcal{S}_{\ell-1} - \{(S_1,\emptyset),(S_2,\emptyset)\})
\cup \{(S_1 \cup S_2,\emptyset)\}.\]
Since \(S_1 \cup S_2 \subseteq S'\), we
have \(\mathcal{S}_{\ell} \preceq \mathcal{S}'\), as required.

At the end of Phase~3, the resulting set pair system
\(\mathcal{S}_{\ell}\) satisfies
\(\mathcal{P}(\mathcal{S}_{\ell}) = \mathcal{P}(\mathcal{S}')\)
and this property is maintained through the last phase.\\

\noindent  \underline{Phase 4}:
There exists \(H \in  \mathcal{H}(\mathcal{S}') - \mathcal{H}(\mathcal{S}_{\ell-1})\). Let \((S,H)\), \((\overline{S},H)\) and \((H,\emptyset)\)
be the set pairs in~\(\mathcal{S}'\) that must exist by~(PL3).
Since \(\mathcal{S}_{\ell-1} \prec \mathcal{S}'\) and Phase~3 is completed,
we have \((S,\emptyset),(\overline{S},\emptyset),\)
\((H,\emptyset) \in \mathcal{S}_{\ell-1}\). Moreover, since Phase~1 is completed,
\(S,\overline{S},H \not \in \mathcal{H}(\mathcal{S}_{\ell-1})\).
We put
\[\mathcal{S}_{\ell-1} \vdash_{(4)} \mathcal{S}_{\ell}
= (\mathcal{S}_{\ell-1} - \{(S,\emptyset),(\overline{S},\emptyset)\})
\cup \{(S,H),(\overline{S},H)\}.\]
By construction, we
have \(\mathcal{S}_{\ell} \preceq \mathcal{S}'\), as required.

At the end of Phase~4, the resulting set pair system
\(\mathcal{S}_{\ell}\) satisfies
\(\mathcal{S}_{\ell} = \mathcal{S}'\). This concludes the
construction of the maximal chain~\(\mathfrak{C}\).
Note that all maximal chains in the
interval~\([\mathcal{S},\mathcal{S}']\) satisfying~(MC)
can be obtained by the construction described above.
The number of iterations in each phase is
completely determined by \(\mathcal{S}\) and \(\mathcal{S}'\).
More specifically, the number of iterations~\(k_1\) in
Phase~1 equals \(|\mathcal{H}(\mathcal{S}) - \mathcal{H}(\mathcal{S}')|\), the number of iterations~\(k_2\) in
Phase~2 equals the number of \((S,\emptyset) \in \mathcal{S}\)
such that there exists some \((S',H) \in \mathcal{S}'\)
with \(S \subseteq S'\) and  \(H \in \mathcal{H}(\mathcal{S}) \cap \mathcal{H}(\mathcal{S}')\)
and the number of iterations~\(k_4\) in
Phase~4 equals \(|\mathcal{H}(\mathcal{S}') - \mathcal{H}(\mathcal{S})|\).
The remaining \(k_3=k-(k_1+k_2+k_4)\) iterations are spent in Phase~3. 
\end{proof}

\begin{figure}
\centering
\includegraphics[scale=1.0]{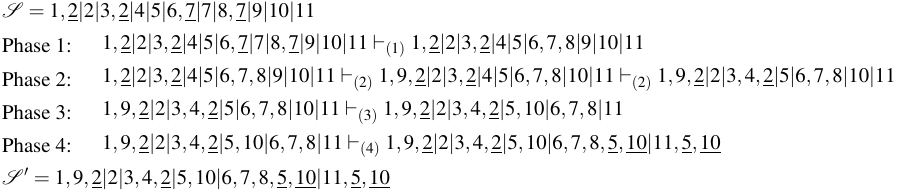}
\caption{An example illustrating the construction
of a maximal chain~\(\mathfrak{C}\) of length~\(k=5\)
with Property~(MC) in the
interval~\([\mathcal{S},\mathcal{S}']\) 
as described in the proof of
Lemma~\ref{lem:max:chain:sorted:by:types}.
In this example we have
\(\mathcal{H}(\mathcal{S}) - \mathcal{H}(\mathcal{S}') = \{\{7\}\}\),
\(\mathcal{H}(\mathcal{S}) \cap \mathcal{H}(\mathcal{S}') = \{\{2\}\}\) and
\(\mathcal{H}(\mathcal{S}') - \mathcal{H}(\mathcal{S}) = \{\{5,10\}\}\).
These sets guide the construction in Phases~1, 2 and 4, respectively.}
\label{fig:mc:property}
\end{figure}

Next we define an edge-labeling~\(\lambda^*\) of \((\mathfrak{P}(X),\preceq)\).
One key aspect is to ensure that any maximal chain in an interval
that is increasing with respect to~\(\lambda^*\) necessarily
satisfies~(MC). In addition, we want to make sure that, among those 
maximal chains in an interval that satisfy~(MC) there is
precisely one that is increasing with respect to~\(\lambda^*\).
The way to achieve the latter is inspired by
the labeling defined in~\cite{wachs1996basis} (see also
\cite[Sec.~3.2.2.]{wachs07})
for the poset of partitions of~\(X\) that was 
illustrated in Figure~\ref{fig:poset:partitions}.
For \((\mathcal{S},\mathcal{S}') \in \mathfrak{E}_{(\mathfrak{P}(X),\preceq)}\), put
\[\lambda^*(\mathcal{S},\mathcal{S}') =
\begin{cases}
\max(S_1 \cup S_2) &\text{if} \ \mathcal{S} \vdash_{(1)} \mathcal{S}',\\[1ex]
n + \max(S_1) &\text{if} \ \mathcal{S} \vdash_{(2)} \mathcal{S}',\\[1ex]
2n + \max(S_1 \cup S_2) &\text{if} \ \mathcal{S} \vdash_{(3)} \mathcal{S}',\\[1ex]
3n + \max(S_3) &\text{if} \ \mathcal{S} \vdash_{(4)} \mathcal{S}'.
\end{cases}\]
By Fact~\ref{fact:edges:p:x} the edge-labeling~\(\lambda^*\)
is well-defined.

As an example, consider the maximal~\(\mathfrak{C}\) of
length~\(5\) in the interval \(\mathfrak{I}=[\mathcal{S},\mathcal{S}']\)
shown in Figure~\ref{fig:mc:property}.
For this chain we obtain the label vector
\(\Lambda_{\lambda^*}(\mathfrak{C}) = (8,20,15,32,43)\).
As can be seen,
\(\mathfrak{C}\) is not increasing with respect to~\(\lambda^*\).
There exists, however, a maximal chain~\(\mathfrak{C}'\)
in the interval~\(\mathfrak{I}\) that can also be
obtained by the construction described in the proof of
Lemma~\ref{lem:max:chain:sorted:by:types}. More
specifically, processing in Phase~2 first~\(4\) and then~\(9\)
we obtain the label
vector \(\Lambda_{\lambda^*}(\mathfrak{C}') = (8,15,20,32,43)\)
and this is the only way to obtain a maximal chain
in~\(\mathfrak{I}\) that is increasing with respect to~\(\lambda^*\).

\begin{theorem}
\label{thm:poset:cactuses:shellable}
The poset \((\mathfrak{P}(X),\preceq)\) is EL-shellable.
\end{theorem}

\begin{proof}
By~\cite[Prop.~1]{huber2024space}, \((\mathfrak{P}(X),\preceq)\)
is bounded and all maximal chains have the same length. Hence,
by definition, \((\mathfrak{P}(X),\preceq)\) is a finite, bounded
and graded poset.

It follows immediately from Lemma~\ref{lem:max:chain:sorted:by:types}
and the definition of~\(\lambda^*\) that \(\lambda^*\) satisfies~(EL1).

It remains to show that~\(\lambda^*\) satisfies~(EL2).
Consider \(\mathcal{S},\mathcal{S}' \in \mathfrak{P}(X)\)
with \(\mathcal{S} \prec \mathcal{S}'\). Put
\(\mathfrak{I} = [\mathcal{S},\mathcal{S}']\).
Let \(\mathfrak{C}\) be the unique maximal chain in~\(\mathfrak{I}\)
that is increasing with respect to~\(\lambda^*\). Consider
a maximal chain~\(\mathfrak{C}'\) in~\(\mathfrak{I}\) with
\(\mathfrak{C} \neq \mathfrak{C}'\).
Put \(k=|\mathfrak{C}|-1=|\mathfrak{C'}|-1\). Let
\[\mathcal{S}_0 \vdash \mathcal{S}_1 \vdash \dots \vdash \mathcal{S}_k 
\quad \text{and} \quad
\mathcal{S}'_0 \vdash \mathcal{S}'_1 \vdash \dots \vdash \mathcal{S}'_k\]
denote the elements of \(\mathfrak{C}\) and
\(\mathfrak{C}'\), respectively.
Since \(\mathfrak{C} \neq \mathfrak{C}'\), we must have
\(k \geq 2\). Moreover, there exists
a smallest \(0 \leq \ell \leq k-2\) such that
\(\mathcal{S}_\ell=\mathcal{S}'_\ell\) and
\(\mathcal{S}_{\ell+1} \neq \mathcal{S}'_{\ell+1}\).
Note that \(\mathfrak{C}'\) need not satisfy~(MC), but by
the construction of~\(\mathfrak{C}\) and the choice
of~\(\ell\), we must have
\(\mathcal{S}_\ell \vdash_{(i)} \mathcal{S}_{\ell+1}\) and
\(\mathcal{S}'_\ell \vdash_{(j)} \mathcal{S}'_{\ell+1}\)
with \(i \leq j\).
Thus, by the construction of~\(\mathfrak{C}\) and the
definition of~\(\lambda^*\), we have
\(\lambda^*(\mathcal{S}_{\ell},\mathcal{S}_{\ell+1}) < 
\lambda^*(\mathcal{S}'_{\ell},\mathcal{S}'_{\ell+1})\), 
implying \(\Lambda_{\lambda^*}(\mathfrak{C}) <_{\text{lex}} \Lambda_{\lambda^*}(\mathfrak{C}')\), as required.
\end{proof}

\section{Equidistant cactus space}
\label{sec:conclusion}

As mentioned in the introduction, part of the motivation 
for this paper was to better understand the space of
equidistant cactuses that was defined in~\cite{huber2024space}.
We give an example of an equidistant cactus in
Figure~\ref{fig:ex:equidistant:cactus}.
The space of equidistant cactuses was constructed
as the \emph{orthant space} corresponding to
the order complex of \((\mathfrak{P},\preceq)\)
(see e.g.~\cite{miller2015polyhedral} for more details about
the construction of an orthant space).
More specifically, each orthant in equidistant cactus space corresponds
to a ranked cactus, and 
the coordinates of a point in the orthant
correspond to the weights assigned to the 
vertices in each rank of the ranked cactus.
For example, in Figure~\ref{fig:ex:equidistant:cactus} 
the ranked cactus corresponds to a 5-dimensional orthant, and the 
weights assigned to the vertices gives the point $(0.06,1.72,0.17,0.39,0.23)$ 
in the orthant that corresponds
to the equidistant cactus in the figure.

Since binary ranked cactuses with leaf set~\(X\)
are in bijection with maximal chains in the poset
\((\mathfrak{P},\preceq)\), the order complex
of the poset \((\mathfrak{P},\preceq)\) with the maximum
\(\{(X,\emptyset)\}\) removed is the so-called \emph{scaffold complex}
of the orthant space of equidistant cactuses
(see e.g.~\cite{miller2015polyhedral}).
Note that the \emph{link of the origin} of an orthant space
is a geometric realization of the scaffold complex 
inside the resulting orthant space. For example, in 
\cite{GD16} the link of the origin of ultrametric tree space is considered.
Since the scaffold complex of this space is the order complex 
of the poset of partitions of~\(X\), it follows
(as pointed out previously in~\cite[p.~48]{ard-kli-06a}),
that (the proper part of) this link
has the homotopy type of a wedge of spheres.
We now show that the link of the origin of equidistant cactus space
enjoys the same property.

We start briefly explaining a technical aspect that also applies to
the space of ultrametric trees mentioned above.
As we are considering equidistant cactuses with a fixed set~\(X\)
of leaves, all orthants of maximum dimension in the
space of equidistant cactuses contain the axis that 
corresponds to the weight assigned to the vertices in rank~0, that is,
assigned to the minimum of the poset~\((\mathfrak{P},\preceq)\).
Having arcs of non-zero length at the leaves is a requirement
motivated from the biological application of these networks.
However, from a mathematical point of view it is more interesting
to look at the part of the space of equidistant cactuses
whose scaffold complex is the proper part of~\((\mathfrak{P},\preceq)\)
which we refer as the \emph{proper part} of the link of 
the origin of the whole space of equidistant cactuses. 
As an immediate consequence of Theorem~\ref{thm:poset:cactuses:shellable}
and \cite[Theorem 3.2.4]{wachs07} we obtain
the following topological result.

\begin{corollary}
\label{cor:homotopy:type:link:origin}
The proper part of the link of the origin the space of equidistant cactuses
has the homotopy type of a wedge of spheres.
\end{corollary}

Note that \cite[Theorem 3.2.4]{wachs07} also asserts that
the number of spheres in the wedge equals the number
of weakly decreasing maximal chains in the underlying
poset. Therefore, it would be interesting to give a formula
for this number.

\begin{figure}
\centering
\includegraphics[scale=1.0]{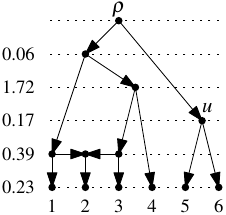}
\caption{An equidistant cactus obtained from the binary
ranked cactus in Figure~\ref{fig:example:binary:cactus}
by assigning a non-negative
real-valued weight to each rank (except the rank of the root~\(\rho\)). This
yields a non-negative distance of each vertex from~\(\rho\) by summing
the weights of all ranks upward from the vertex. Vertex~\(u\), for example,
has distance~\(0.17+1.72+0.06=1.95\) from~\(\rho\). All leaves
have the same distance from~\(\rho\), hence the name equidistant cactus.}
\label{fig:ex:equidistant:cactus}
\end{figure}

We conclude by noting that in~\cite{moulton2025spaces} a space of
\emph{ranked tree-child networks} is introduced which can be
thought of as a generalization of the space of equidistant cactuses. 
An encoding of ranked tree-child networks is described in~\cite{moulton2025spaces} that is based also 
on a poset \((\mathfrak{T}(X),\preceq)\)
whose ground set consists of collections of subsets of~\(X\).
We remark that the poset \((\mathfrak{P}(X),\preceq)\) can be 
isomorphically embedded into \((\mathfrak{T}(X),\preceq)\) by
mapping \(\mathcal{S} \in \mathfrak{P}(X)\) to
\(\{S \cup H : (S,H) \in \mathcal{S}\} - \mathcal{H}(\mathcal{S})\).
With this in mind it would be interesting to know whether or not 
the poset \((\mathfrak{T}(X),\preceq)\) is also shellable.

\subsubsection*{Acknowledgments}
V. Moulton and A. Spillner thank the organizers of the
workshop ``New Directions in Experimental Mathematics''
in Bielefeld 2025 where they started to
discuss the shellability of the poset~\((\mathfrak{P}(X),\preceq)\).
Part of this work was funded by the Land Sachsen-Anhalt and
the European Union through the European Social Fund Plus
(ESF+), grant no. ZS/2024/10/191230.

\bibliographystyle{plain}
\bibliography{network}

\end{document}